\newcommand{\dataversione}{October 20, 2017}

\documentclass[11pt,reqno,a4paper,final]{amsart}
\usepackage{fancyhdr}
\usepackage{mathrsfs, amssymb, stmaryrd}

\usepackage[notref,notcite]{showkeys}
\usepackage{hyperref}\hypersetup{colorlinks=true, citecolor=blue}

\numberwithin{equation}{section}

\newtheoremstyle{mytheorem}
{10pt}
{10pt}
{\it}
{\parindent}
{\bf}
{}
{ }
{\thmname{#1}\thmnumber{~#2.}\thmnote{~\rm#3}}

\newtheoremstyle{myremark}
{10pt}
{10pt}
{\rm}
{\parindent}
{\bf}
{}
{ }
{\thmname{#1}\thmnumber{~#2.}\thmnote{~\rm#3}}

\newtheoremstyle{myparagraph}
{10pt}
{10pt}
{\rm}
{\parindent}
{\bf}
{.}
{ }
{\thmnumber{#2.~}\thmname{#1}\thmnote{#3}}

\theoremstyle{mytheorem}
\newtheorem{theorem}[subsection]{Theorem}
\newtheorem{lemma}[subsection]{Lemma}

\newtheorem{proposition}[subsection]{Proposition}

\theoremstyle{myremark}
\newtheorem{remark}[subsection]{Remark}

\theoremstyle{myparagraph}
\newtheorem{parag}[subsection]{}
\newtheorem*{parag*}{}

\makeatletter
\def\@secnumfont{\sc}
\def\section{\@startsection%
{section}
{1}
\z@{1.5\linespacing\@plus .2\linespacing}
  {.7\linespacing}
  {\normalfont\sc\centering}}
\makeatother

\makeatletter
\renewenvironment{proof}[1][\proofname]{\par 
  \pushQED{\qed}%
  \normalfont \topsep10\p@\@plus6\p@\relax 
  \trivlist 
  \itemindent\normalparindent 
  \item[\hskip\labelsep 
    \bfseries 
    #1\@addpunct{.}]\ignorespaces 
}{%
  \popQED\endtrivlist\@endpefalse 
} 
\providecommand{\proofname}{Proof}
\makeatother

\newcommand{\footnoteb}[1]{\footnote{~#1}}

\newenvironment{itemizeb}
{\begin{itemize}\itemsep=3 pt\leftskip=0 pt\labelsep=5 pt}
{\end{itemize}}

\newcommand{\R}{\mathbb{R}}
\newcommand{\HH}{\mathbb{H}}
\newcommand{\Mass}{\mathbb{M}}
\newcommand{\Haus}{\mathscr{H}}
\newcommand{\Leb}{\mathscr{L}}
\newcommand{\Tan}{\mathrm{Tan}}
\newcommand{\Span}{\mathrm{span}}
\newcommand{\bd}{\partial}
\newcommand{\de}{\mathrm{d}}
\newcommand{\miniskip}{\vskip 3 pt}
\newcommand{\IC}[1]{\llbracket#1\rrbracket}
\newcommand{\scalar}[1]{\langle#1\rangle}
\newcommand{\bigscalar}[1]{\big\langle#1\big\rangle}

\begin{document}


\thispagestyle{empty}

~\vskip -1.1 cm
	%
	%
{\footnotesize\noindent 
[version: \dataversione]
\hfill \textit{Rend.\ Lincei Mat.\ Appl.} 28 (2017), no.~4, 861-869
\par
\hfill DOI~\href{http://dx.doi.org/10.4171/RLM/788}{10.4171/RLM/788} \par
}

\vspace{1.7 cm}

	%
	%
{\centering\Large\bf
On some geometric properties of currents
\vskip 4 pt
and Frobenius theorem
\\
}

\vspace{.8 cm}

	%
	%
{\centering\sc Giovanni Alberti, Annalisa Massaccesi\\}

\vspace{.8 cm}

{\rightskip 1 cm
\leftskip 1 cm
\parindent 0 pt
\footnotesize
	%
	%
{\sc Abstract.}
In this note we announce some results,
due to appear in \cite{AlMas}, \cite{AlMasSte},
on the structure of integral and normal 
currents, and their relation to
Frobenius theorem.
In particular we show
that an integral current cannot be tangent
to a distribution of planes which is nowhere
involutive (Theorem~\ref{s-frobenius1}), 
and that a normal current which is tangent 
to an involutive distribution of planes
can be locally foliated in terms
of integral currents (Theorem~\ref{s-frobenius2}).
This statement gives a partial
answer to a question raised by Frank Morgan
in \cite{GMT}.


\par
}

\section{Introduction}
\label{s1}
Consider a distribution of $k$-dimensional planes
in $\R^n$, namely a map that associates to each point 
$x\in\R^n$ a $k$-dimensional subspace $V(x)$ of $\R^n$, 
and assume that $V$ is spanned by vectorfields
$v_1,\dots,v_k$ of class $C^1$.
We say that $V$ is \emph{involutive}
at a point $x\in\R^n$ if the commutators 
of the vectorfields $v_1,\dots,v_k$, 
evaluated at $x$, belong to $V(x)$ (see~\S\ref{s-vectorfields}).
Moreover, given a $k$-dimensional surface $S$ in $\R^n$, 
we say that $S$ is \emph{tangent} to $V$ 
if the tangent space $\Tan(S,x)$ agrees with $V(x)$ 
for every $x\in S$.

In this context, 
the first part of Frobenius theorem
states that, if $S$ is tangent to $V$, then $V$ 
must be involutive at every point of $S$.
Or, in a slightly weaker form, that if $V$ is nowhere involutive
then there exist no tangent surfaces (cf.~\cite{Lee}, Theorem~14.5).

The classical version of this theorem requires 
that the surface $S$ is at least of class $C^1$, and 
it is then natural to ask if similar 
statements hold for weaker notions of surface. 
To this regard, we mention that a positive answer 
for \emph{Sobolev surfaces}, that is, 
Sobolev images of open subsets of $\R^{2h}$,  
has been given in \cite{MaMaMo}, Theorem~1.2, 
when $V$ is the distribution of $2h$-planes
in $\R^{n=2h+1}$ corresponding to the horizontal
distribution in the sub-Riemannian Heisenberg group $\HH^h$.

In section~\ref{s3} we give a positive answer for 
\emph{integral currents},%
\footnoteb{The basic definitions and terminology concerning currents 
are recalled in Section~\ref{s2}.}
and more precisely we show that
given an integral $k$-dimensional current $T$
which is tangent to $V$, 
then $V$ must be involutive on the support of $T$
(Theorem~\ref{s-frobenius1}).
Note that the assumption that $T$ is integral is crucial, 
and indeed the analogous statement for rectifiable sets 
does not hold, cf.~Remark~\ref{s-frobenius1rem}(a).

It turns out that Theorem~\ref{s-frobenius1} 
is an immediate consequence 
of the following geometric property of the boundary of 
integral currents, which is actually the heart
of the matter: 
if $T$ is an integral $k$-dimensional current tangent to a 
continuous distribution of $k$-planes $V$, 
then $\bd T$ is tangent to $V$ as well
(see \S\ref{s-tangent} for the definition of tangency, 
and Theorems~\ref{s-geopro1} and \ref{s-geopro2}).

\miniskip
In Section~\ref{s4} we turn to the other part of Frobenius theorem, 
which states that if $V$ is everywhere involutive, then $\R^n$ can be locally
foliated with $k$-dimensional surfaces which are tangent 
to $V$. 
In Theorem~\ref{s-frobenius2} we prove the 
following generalization:
if $V$ is everywhere involutive and $T$ is a $k$-dimensional 
\emph{normal current} tangent to $V$, then $T$ can be locally foliated 
by a family $k$-dimensional integral currents tangent to $V$
(the definition of foliation, or mass decomposition, 
of a current is given in \S\ref{s-foliation}).
Conversely, if $T$ can be foliated then $V$ must be 
involutive at every point in the support of $T$.

The first part of Theorem~\ref{s-frobenius2} gives a partial 
positive answer to a question raised by Frank Morgan 
in \cite{GMT}, namely if every normal current admits a foliation
in terms of integral currents
(other positive results were given in \cite{HP}, \cite{PaoSte}, 
\cite{Smirnov}, see Remark~\ref{s-frobenius2rem}). 
On the other hand, the second part shows that
a normal current which is tangent to a nowhere involutive 
distribution of planes admits no foliation 
of a certain type: this result was first stated in 
\cite{Zworski}, but in a form which is not correct
(see Remark~\ref{s-frobenius2rem}(c) for more details).

\section{Notation}
\label{s2}
In this section we briefly recall some notation and basic definitions.
For rectifiable sets and currents we essentially follow \cite{KP}. 
As usual, $\Haus^k$ stands for the $k$-dimensional Hausdorff measure
and $\Leb^n$ for the Lebesgue measure on $\R^n$.

\medskip
In the following we fix an open set $\Omega$ in $\R^n$.

\begin{parag}[The vectorfield $\boldsymbol{v}$
and the distribution of planes $\boldsymbol{V}$]
\label{s-vectorfields}
In the following we consider 
$v_1,\dots,v_k$ continuous vectorfields on $\Omega$
with $0<k<n$, and the simple $k$-vectorfield
\[
v := v_1\wedge\dots\wedge v_k
\, .
\]
Moreover we assume that $v$ is unitary, 
that is, $|v(x)|=1$ for every $x\in\Omega$, 
and we denote by $V$ the distribution of $k$-planes 
spanned by $v$, that is, 
\[
V(x):= \Span(v(x)) := \Span\big\{v_1(x),\dots,v_k(x)\big\}
\quad\text{for every $x\in\Omega$.}
\]
With a slight abuse of language, 
we say that $V$ (or $v$) is of class $C^1$ to mean that 
$v_1,\dots,v_k$ are of class $C^1$,
and if this is the case we say that $V$ is \emph{involutive} at 
a point $x\in\Omega$ if
\[
[v_i,v_j](x) \in V(x)
\quad\text{for every $1\le i,j\le k$,}
\]
where $[v_i,v_j]$ is the Lie bracket, or commutator, of $v_i$ and $v_j$.%
\footnoteb{That is, the vectorfield defined by
\[
[v_i,v_j](x)
:= \bigscalar{\nabla v_j(x); v_i(x)} 
   - \bigscalar{\nabla v_i(x); v_j(x)} 
 = \frac{\bd v_j}{\bd v_i}(x) 
   - \frac{\bd v_i}{\bd v_j}(x)
\, ,
\]
where $\scalar{{~;~}}$ denotes the usual pairing of matrices and vectors.}
\end{parag}

\begin{parag}[Rectifiable sets, orientation]
\label{s-rectifiable} 
A set $\Sigma$ in $\Omega$ is \emph{rectifiable} of 
dimension $k$, or $k$-rectifiable,
if it has finite $\Haus^k$ measure and can be covered, 
except for an $\Haus^k$-null subset, by countably many 
surfaces of dimension $k$ and class $C^1$.%
\,\footnoteb{Through this paper 
sets, maps and vectorfields are always 
(at least) Borel measurable.}

Then at $\Haus^k$-a.e.~$x\in\Sigma$ there exists 
an \emph{approximate tangent space} $\Tan(\Sigma,x)$, 
which is characterized (for $\Haus^k$-a.e.~$x\in\Sigma$) 
by the following property:
for every $k$-surface $S$ of class $C^1$ there holds
\[
\Tan(\Sigma,x)=\Tan(S,x)
\quad\text{for $\Haus^k$-a.e.~$x\in\Sigma\cap S$.}
\]
An \emph{orientation} of $\Sigma$ is a simple $k$-vectorfield 
$\tau$ defined on $\Sigma$ such that $\tau(x)$ spans 
$\Tan(\Sigma,x)$ and has norm $1$ for 
$\Haus^k$-a.e.~$x\in\Sigma$.
\end{parag}

\begin{parag}[Currents, boundary, mass, normal currents]
\label{s-currents} 
A $k$-dimensional current, or $k$-current, in $\Omega$ is a (continuous) 
linear functional on the space of smooth $k$-forms with compact 
support on $\Omega$. The \emph{boundary} of a $k$-current $T$ is the 
$(k-1)$-current $\bd T$ defined by 
$\scalar{\bd T;\omega} := \scalar{T;\de\omega}$,
where $\de\omega$ is the exterior differential
of the form $\omega$.

The \emph{mass} of $T$, denoted by $\Mass(T)$, 
is the supremum of $\scalar{T;\omega}$ over all forms $\omega$
such that $|\omega(x)| \le 1$ for every~$x$.
A current $T$ with finite mass can be represented as a 
vector measure, that is, there exist a positive finite measure 
$\mu$ on $\Omega$ and a map $\tau$ from $\Omega$ to the set of
$k$-vectors with norm $1$, called \emph{orientation}, such that
\[
\scalar{T;\omega} 
:=\int_{\Omega} \bigscalar{\tau(x);\omega(x)} \, d\mu(x)
\, ,
\]
where $\scalar{{~;~}}$ is the usual pairing of 
$k$-vectors and $k$-covectors.
In this case we simply write $T=\tau\mu$.
Note that the mass of $T$ is $\Mass(T)=\mu(\Omega)=\|\mu\|$.

A current $T$ is called \emph{normal} 
if both $T$ and $\bd T$ have finite mass.
\end{parag}

\begin{parag}[Rectifiable and integral currents]
\label{s-intcurrents} 
A $k$-current $T$ is called 
\emph{rectifiable} (with integral multiplicity)
if there exist a $k$-rectifiable set $\Sigma$, 
an orientation $\tau$ of $\Sigma$, 
and a \emph{positive}, integer-valued multiplicity 
$\theta\in L^1(\Sigma,\Haus^k)$ such that
\[
\scalar{T;\omega} 
:=\int_\Sigma \bigscalar{\tau(x);\omega(x)} \, \theta(x) \, d\Haus^k(x)
\, .
\]
In this case we write $T=\IC{\Sigma,\tau,\theta}$. 
Note that the mass of $T$ agrees with the $k$-dimensional 
measure of $\Sigma$ counted with multiplicity, that is, 
$\Mass(T)=\int_\Sigma \theta \, d\Haus^k$;
in particular $\Mass(T)$ is finite.

A current $T$ is called \emph{integral} if both $T$ and $\bd T$ 
are rectifiable; in particular every integral current is normal.
\end{parag}

\begin{parag}[Notions of tangency]
\label{s-tangent} 
Take $v$ and $V$ as in \S\ref{s-vectorfields}
We say that an $h$-rectifiable set $\Sigma$ with $h\le k$
is tangent to $V$ if the tangent space $\Tan(\Sigma,x)$
is contained in $V(x)$ for $\Haus^h$-a.e.~$x\in\Sigma$.

Accordingly, a rectifiable 
$h$-current $T=\IC{\Sigma,\tau,\theta}$
is tangent to $V$ if the supporting rectifiable 
set $\Sigma$ is so.
More generally, an $h$-current with finite mass 
$T=\tau\mu$ is tangent to $V$ if the span of 
the $h$-vector $\tau(x)$ is contained in $V(x)$ 
for $\mu$-a.e.~$x$.%
\footnoteb{The span of a $h$-vector $w$ in $\R^n$ is defined 
as the smallest subspace $W$ of $\R^n$ such that $w$ is also 
a $h$-vector in $W$. If $w$ is a simple vector we 
recover the usual definition.}

Moreover we say that a rectifiable $k$-current 
$T=\IC{\Sigma,\tau,\theta}$
is \emph{oriented} by $v$ if $\tau(x)=v(x)$ 
for $\Haus^k$-a.e.~$x\in\Sigma$, and more
generally, a $k$-current with finite mass 
$T=\tau\mu$ is oriented by $v$ if $\tau(x)=v(x)$
for $\mu$-a.e.~$x$.
\end{parag}

\begin{remark}
\label{s-tangentrem} 
If $T=\tau\mu$ is a $k$-current with finite mass, 
then $T$ is tangent to $V$ if and only if
$\tau(x)=\pm v(x)$ for $\mu$-a.e.~$x$
(recall that $v$ is unitary).
In particular if $T$ is oriented by 
$v$ then it is also tangent to $V$, but clearly 
the converse does not hold. 
\end{remark}


	%
	%
\section{Geometric structure of the boundary}
\label{s3} 
Through this section, $v$ and $V$ 
are taken as in \S\ref{s-vectorfields}.

The next two statements are the main results in this section, 
and establish a natural (and apparently obvious)
relation between the tangent space of a current $T$ and the 
tangent space of the boundary $\bd T$, namely that, 
under suitable assumptions, the former contains the latter.

\begin{theorem}[(See~\cite{AlMas}.)]
\label{s-geopro1}
If $T$ is an integral $k$-current
oriented by $v$, then the boundary $\bd T$ 
is tangent to~$V$.
\end{theorem}

\begin{theorem}[(See~\cite{AlMasSte}.)]
\label{s-geopro2}
If $V$ is of class $C^1$ and $T$ 
is an integral $k$-current tangent to~$V$, 
then $\bd T$ is tangent to~$V$.
\end{theorem}

\begin{remark}
\label{s-geoprorem}
(a)~Theorem~\ref{s-geopro2} can be viewed
as the ``non-oriented version'' of Theorem~\ref{s-geopro1}, 
and under the assumption that $V$ is of class $C^1$ 
it is actually a stronger statement 
(cf.~Remark~\ref{s-tangentrem}). 

\miniskip
(b)~Theorem~\ref{s-geopro1} can be proved in a slightly 
stronger form, and under slightly weaker assumptions
on the current $T$ (see \cite{AlMas} for more details); 
the key step of the proof consists in taking the 
blow-up of $T$ at ``almost every point of the boundary'', 
and here the assumption that $T$ is integral (or slightly less)
plays an essential role.

\miniskip
(c)~Theorem~\ref{s-geopro2} can be proved 
under much weaker assumptions on the current $T$, 
including the case where $T$ is normal 
and $\bd T$ is singular with respect to~$T$.%
\footnoteb{Here both $T$ and $\bd T$ are viewed as 
(vector-valued) measures.}
The proof is completely different from that of
Theorem~\ref{s-geopro1}, and relies heavily 
on the fact that $V$ is of class $C^1$ (see \cite{AlMasSte}).
Note that this regularity assumption on $V$ can perhaps be weakened, 
but cannot be entirely dropped: indeed in \cite{AlMas} we 
construct a continuous distribution $V$ of
$2$-planes in $\R^3$ and an integral $2$-current 
$T$ such that $T$ is tangent to $V$ but $\bd T$ is not.%
\footnoteb{This current is actually (supported
on) the graph of a continuous Sobolev function.}

\miniskip
(d)~In \cite{AlMasSte} we also show that
if $V$ is everywhere involutive 
then Theorem~\ref{s-geopro2} holds for every 
normal $k$-current $T$. This is no longer true 
if $V$ is not everywhere involutive,
the counterexample being any current on $\Omega$ 
of the form $T:=v\mu$ where $\mu:=\rho\Leb^n$
and $\rho$ is a function of class $C^1$ 
whose support is compact and contained in the 
(open) set of all points where $v$ is not involutive.
\end{remark}

The relation between the geometric property of the boundary 
of $T$ proved in Theorem~\ref{s-geopro2} and Frobenius theorem 
is made clear in the following statement.

\begin{proposition}
\label{s-basic}
Assume that $V$ is of class $C^1$,  
and let $T$ be a normal $k$-current tangent to $V$ such that
$\bd T$ is also tangent to~$V$.
Then $V$ is involutive at every point of the support of $T$.%
\footnoteb{By support of a current with finite mass $T=\tau\mu$
we mean the support of the measure $\mu$, that is, the smallest 
closed set $F$ such that $\mu(\R^n\setminus F)=0$.
If $T$ is rectifiable, that is $T=\IC{\Sigma,\tau,\theta}$, 
then the support of $T$ turns out to be the closure of the set of 
all points $x\in\Sigma$ where the $k$-dimensional density of $\Sigma$ 
is~$1$.}
\end{proposition}

This result is an immediate consequence of the following
lemma:

\begin{lemma}[(See~\cite{AlMas}.)]
\label{s-basiclemma}
If $V$ is of class $C^1$ and is not involutive
at a point $x_0\in\Omega$, then there exists 
a $(k-1)$-form $\alpha$ of class $C^1$ on $\Omega$
such that
\begin{itemizeb}
\item[(i)] 
for every $x\in\R^n$ the restriction of $\alpha(x)$ 
to $V(x)$ is zero;%
\,\footnoteb{That is, $\scalar{w;\alpha(x)}=0$
for every ($k-1$)-vector $w$ 
whose span is contained in $V(x)$.}
\item[(ii)] 
$\scalar{v(x_0);\de\alpha(x_0)} \ne 0$.
\end{itemizeb}
\end{lemma}

\begin{proof}[Proof of Proposition~\ref{s-basic}]
We write $T=\tau\mu$, 
and we assume by contradiction that there
exists a point $x_0$ in the support of $\mu$
where $v$ is not involutive.


We take $\alpha$ as in Lemma~\ref{s-basiclemma}.
Then for every smooth function $\varphi$ with compact 
support on $\Omega$ there holds
\begin{align*}
  0 
  = \scalar{\bd T;\varphi\alpha}
  = \scalar{T;\de(\varphi\alpha)}
& = \scalar{T;\de\varphi\wedge\alpha} + \scalar{T;\varphi\,\de\alpha} \\
& = \scalar{T;\varphi\,\de\alpha}
  = \int_{\Omega} \scalar{\tau;\de\alpha} \, \varphi \, d\mu
\end{align*}
(the first equality follows from the fact that $\bd T$ is tangent 
to $V$ and property~(i) in Lemma~\ref{s-basiclemma};
the third one from the identity 
$\de(\varphi\alpha) = \de\varphi\wedge\alpha+\varphi\,\de\alpha$;
the fourth one by the fact that $T$ is tangent to $V$ and
the restriction of the $k$-form $\de\varphi\wedge\alpha$ 
to $V$ is null, again by property~(i) in Lemma~\ref{s-basiclemma}).

Since $\varphi$ is arbitrary 
we infer that $\scalar{\tau;\de\alpha}=0$ $\mu$-a.e., and since 
$\tau=\pm v$ (because $T$ is tangent to $V$, 
cf.~Remark~\ref{s-tangentrem}) we obtain 
that $\scalar{v;\de\alpha}=0$ $\mu$-a.e.

On the other hand, property~(ii) in Lemma~\ref{s-basiclemma}  
implies that $\scalar{v;\de\alpha} \ne 0$ in a neighbourhood 
of $x_0$. 
Since $x_0$ is in the support of $\mu$,
this neighbourhood has positive $\mu$ measure, and
we have a contradiction.
\end{proof}

Using Theorem~\ref{s-geopro2} and Proposition~\ref{s-basic}
we immediately obtain the following:

\begin{theorem}
\label{s-frobenius1}
Assume that $V$ is of class $C^1$ and that $T$ is an
integral $k$-current tangent to $V$. Then 
$V$ is involutive at every point in the support of $T$.
\end{theorem}

\begin{remark}
\label{s-frobenius1rem}
(a)~The analogue of Theorem~\ref{s-frobenius1} 
for rectifiable sets does not hold. 
Indeed in \cite{AlMas}
we show that for every distribution $V$, even 
a nowhere involutive one, 
it is possible to find a $k$-dimensional surface $S$
of class $C^1$ whose tangency set
\[
\Sigma:=
\big\{ 
  x\in S \, : \, \Tan(S,x)=V(x)
\big\}
\]
has positive $\Haus^k$-measure;
in particular $\Sigma$ is a non-trivial
$k$-rectifiable set tangent 
to~$V$. (This result was first proved 
in a slightly less general form in 
\cite{Balogh}, Theorem~1.4.) 

\miniskip
(b)~Using Theorem~\ref{s-frobenius1} we can partly recover
(and even extend) the Frobenius theorem for 
Sobolev surfaces proved in \cite{MaMaMo}, Theorem~1.2.
To be precise, by Sobolev surface  
we mean a $k$-rectifiable set $\Sigma$ of the form 
$\Sigma=f(U)$ where $U$ is an open set in $\R^k$ 
and $f:U\to\Omega$ is a continuous map of class 
$W^{1,p}$ with $p>k$, and we can show the following
(see~\cite{AlMas}): if $V$ is of class $C^1$ and 
$\Sigma$ is a Sobolev surface tangent to $V$, then 
$V$ is involutive at $\Haus^k$-a.e.~point of $\Sigma$.
\end{remark}

\section{Foliations of normal currents}
\label{s4}
We begin this section by giving the definition 
or foliation of a current, and then we show
that for a normal current which is tangent 
to a distribution of planes $V$ of class $C^1$ 
the existence of a foliation is strictly related to the 
involutivity of $V$ (Theorem~\ref{s-frobenius2}).

\begin{parag}[Foliations of currents]
\label{s-foliation} 
Let $T=\tau\mu$ be a $k$-current with finite mass in $\Omega$, 
and let $\{R_t\}$ be a family of rectifiable $k$-currents
in $\Omega$, where $t$ varies in some index space $I$ endowed 
with a measure $dt$.%
\footnoteb{We also assume that the function
$t\mapsto\Mass(R_t)$ and $t\mapsto\scalar{R_t;\omega}$
are Borel measurable for every $k$-form $\omega$
on $\Omega$ of class $C^\infty_c$ 
(or, equivalently, of class $C_0$).}
We say that $\{R_t\}$ is a \emph{mass decomposition}, 
or \emph{foliation}, of $T$ if 
\begin{itemizeb}
\item[(i)]
$\scalar{T;\omega} = \int_I \scalar{R_t;\omega} \, dt$
for every $k$-form $\omega$ on $\Omega$ of class $C^\infty_c$;
\item[(ii)]
$\Mass(T) = \int_I \Mass(R_t) \, dt$.
\end{itemizeb}
If $T$ is a normal current, we may also consider 
the following additional conditions:
\begin{itemizeb}
\item[(iii)]
$\int_I \Mass(\bd R_t) \, dt <+\infty$;
\item[(iv)]
$\Mass(\bd T) = \int_I \Mass(\bd R_t) \, dt$.
\end{itemizeb}
\end{parag}

\begin{remark}
\label{s-foliationrem}
(a)~Condition~(i) is often written in compact form:
$T = \int_I R_t \, dt$.

\miniskip
(b)~If $T$ is oriented by a continuous $k$-vectorfield $v$, 
then condition~(ii) implies that $R_t$ is oriented 
by $v$ for a.e.~$t\in I$.%
\footnoteb{Conversely, if $\int_I\Mass(R_t) \, dt$ is finite, 
(i)~holds, and $R_t$ is oriented by $v$ for a.e.~$t$, 
then (ii)~holds.}
This explains the term ``foliation''.

\miniskip
(c)~Conditions~(i) and (iii) imply that
$\bd T=\int_I \bd R_t \, dt$.
Condition~(iv) is stronger than (iii), and implies that the family
$\{\bd R_t\}$ is a foliation of $\bd T$.

\miniskip
(d)~A current of finite mass $T=\tau\mu$ may admit no foliation. 
For example this happens if $\mu$ is a Dirac mass, 
or more generally a measure supported on a set $E$ 
which is purely $k$-unrectifiable.%
\footnoteb{That is, $\Haus^k(E\cap\Sigma)=0$ for every $k$-rectifiable
set $\Sigma$.}
Or if $\mu$ is the restriction of $\Haus^k$ to a $k$-surface 
$S$ but $\tau$ does not span the tangent bundle of~$S$.%
\footnoteb{The point is that for currents with finite mass 
the measure $\mu$ can be chosen independently of 
the orientation~$\tau$.
This is not the case with normal currents, and indeed
none of these examples is a normal current.}
\end{remark}

While the question of the existence of foliations for 
currents with finite mass is not particularly interesting, 
the same question for normal currents is quite relevant, 
and was first formulated by Frank Morgan
(see \cite{GMT}, Problem~3.8).
The next result answers this question 
for normal currents which are tangent to a distribution 
of planes of class $C^1$.
If no regularity assumption is made on the tangent bundle 
of the currents there are a few partial results 
(see Remark~\ref{s-frobenius2rem}) and  
the question is not completely settled.

\begin{theorem}[(See~\cite{AlMas}.)]
\label{s-frobenius2}
Let $V$ be a distribution of $k$-planes of class $C^1$ on $\Omega$.

\miniskip 
\emph{(i)}~If $V$ is everywhere involutive, then every point of $\Omega$ 
admits a neighbourhood $U$ such that every normal $k$-current 
in $U$ tangent to $V$ admits a foliation $\{R_t\}$
satisfying conditions~(i), (ii), (iv) in \S\ref{s-foliation}.

\miniskip
\emph{(ii)}~Conversely, if $T$ is a normal $k$-current in $\Omega$ 
which is tangent to $V$ and admits a foliation $\{R_t\}$ 
satisfying conditions~(i), (ii) in \S\ref{s-foliation} 
and such that the currents $R_t$ are integral,
then $V$ is involutive at every point in the support of $T$.
\end{theorem}

\begin{remark}
\label{s-frobenius2rem}
(a)~Statement~(ii) is an immediate consequence of Theorem~\ref{s-frobenius1}. 

\miniskip
(b)~Statement~(ii) shows that the answer to Morgan's 
question is negative whenever $1<k<n$, 
at least if we require that the 
currents in the foliation are integral 
(and not just rectifiable); the example is given by
any normal current $T$ tangent to distribution 
of $k$-planes which is nowhere involutive, 
for example the normal $k$-current $T$ given in 
Remark~\ref{s-geoprorem}(d).

\miniskip
(c)~A negative answer to Morgan's question 
was first given by M.~Zworski, 
who proposed the following variant of 
statement~(ii) above (see \cite{Zworski}, Theorem~2):
if $v$ is a nowhere involutive $k$-vectorfield
and $T$ is a current of the form $T=v\Leb^n$, 
then $T$ admits no foliation. 
However this statement is not correct, because
it does not require that the currents in the foliation 
are integral, and for $k=n-1$ it contradicts the fact
that every normal current admits a foliation, 
see remark~(e) below.

\miniskip
(d)~Every normal $1$-current in $\Omega$ admits a foliation 
satisfying conditions~(i), (ii), (iii) in \S\ref{s-foliation};
this is essentially a consequence of the decomposition 
result by S.~Smirnov \cite{Smirnov} (see also \cite{PaoSte}), 
even though it is not explicitly stated there.
This result does not hold if we 
require that condition~(iv) holds. 

\miniskip
(e)~Consider a normal $(n-1)$-current $T$ in $\Omega$.
A consequence of the coarea formula for $BV$ 
functions is that if $T$ is a boundary then it 
admits a foliation satisfying 
conditions~(i), (ii), (iv) in \S\ref{s-foliation}
(see \cite{Federer}, Theorem~4.5.9(13)).
Such a foliation exists also if the boundary of $T$ 
is rectifiable, as proved in \cite{Zworski}, Theorem~1, using 
an idea from \cite{HP}. 
By modifying the argument in \cite{HP}
we prove in \cite{AlMas} that that every normal 
$(n-1)$-current admits a foliation.%
\footnoteb{In this case the currents
in the foliation are no better than rectifiable.}

\miniskip
(f)~The existence of foliations for normal currents
of dimension $d=1$ or $d=n-1$ mentioned in items~(d) and (e) 
above has no counterpart for $2\le d\le n-2$.
Indeed, for any $n\ge 4$, Andrea Schioppa constructed in 
\cite{Schioppa} a normal current $T$ of codimension $2$ 
in $\R^n$ whose support is purely $2$-unrectifiable. 
Clearly such $T$ admits no foliation, and more precisely it
cannot even be decomposed as $T=\int_I R_t \, dt$
with the only assumption that $\int_I \Mass(R_t)\, dt$
is finite.

\end{remark}

\section*{Acknowledgements}
Part of this research 
was carried out while the second author was visiting
the Mathematics Department in Pisa, 
supported by the University of Pisa through the 2015 PRA Grant
``Variational methods for geometric problems''.   
The research of the first author has been partially supported
by the Italian Ministry of Education, University and Research (MIUR) 
through the 2011 PRIN Grant ``Calculus of variations'', and by
the European Research Council (ERC) through the 2011 Advanced Grant 
``Local structure of sets, measures and currents''.
The research of the second author has been partially supported
by the European Research Council through the 2012 Starting Grant 
``Regularity theory for area minimizing currents''. 

	%
	%
	%
	%
\bibliographystyle{plain}

	%
	%
	%
	%
\vskip .5 cm

{\parindent = 0 pt\footnotesize
G.A.
\\
Dipartimento di Matematica, 
Universit\`a di Pisa
\\
largo Pontecorvo~5, 
56127 Pisa, 
Italy 
\\
\texttt{giovanni.alberti@unipi.it}

\bigskip
A.M.
\\
Institut f\"ur Mathematik,
Universit\"at Z\"urich
\\
Winterthurerstrasse 190,
8057 Z\"urich,
Switzerland
\\
\texttt{annalisa.massaccesi@math.uzh.ch}

}

\end{document}